\documentclass[a4paper]{amsart}
\usepackage{amssymb}
\usepackage[T1]{fontenc}

\usepackage[colorlinks=true, linkcolor=blue, citecolor=red] {hyperref}
\usepackage{graphicx}
\allowdisplaybreaks

\numberwithin{equation}{section}

\newtheorem{theorem}{Theorem}
\newtheorem{lemma}[theorem]{Lemma}

\newtheorem{definition}[theorem]{Definition}
\newtheorem{remark}[theorem]{Remark}

\newcommand{\eps} {\varepsilon}
\newcommand{\test}{\varphi}

\DeclareMathOperator{\sign}{sign}
\newcommand{\sgn}[1]{\sign\left(#1\right)}
\newcommand{\Dt}{{\Delta t}}
\newcommand{\Dx}{{\Delta x}}

\newcommand{\norm}[1]{\left\| #1 \right\|}
\newcommand{\Dm}{D_-}

\newcommand{\Dpt}{D^t_+}
\newcommand{\Dmt}{D^t_-}
\newcommand{\abs}[1]{\left|#1\right|}
\newcommand{\dm}{\Delta_-}

\newcommand{\mi}[2]{{#1}\wedge {#2}}
\newcommand{\mx}[2]{{#1}\vee {#2}}
\newcommand{\lipnorm}[1]{\norm{#1}_{\mathrm{Lip}}}
\newcommand{\seq}[1]{\left\{#1\right\}}
\begin{document}
\title[The OH-equation with space dependent flux]{The Ostrovsky-Hunter
  equation with a space dependent flux function}

\author[N. Chatterjee]{N. Chatterjee}
\address[Neelabja Chatterjee]
{\newline Department of mathematics
\newline University of Oslo
\newline P.O. Box 1053,  Blindern
\newline N--0316 Oslo, Norway} 
\email[]{neelabjc@math.uio.no}

\author[N. H. Risebro]{N. H. Risebro}
\address[Nils Henrik Risebro]
{\newline Department of mathematics
\newline University of Oslo
\newline P.O. Box 1053,  Blindern
\newline N--0316 Oslo, Norway} 
\email[]{nilsr@math.uio.no}

\date{\today}



\thanks{This project has received funding from the European Union's
Horizon 2020 research and innovation programme under the Marie Sk\l{}odowska-Curie grant
agreement No 642768.}

\begin{abstract}
  We study the periodic Ostrovsky-Hunter equation in the case where the flux
  function may depend on the spatial variable. Our main results are
  that if the flux function is twice differentiable, then there exists
  a unique entropy solution. This entropy solution may be constructed
  as a limit of approximate solutions generated by a finite volume
  scheme, and the finite volume approximations converge to the entropy
  solution at a rate $1/2$.
\end{abstract}

\maketitle

\section{Introduction}

To model small-amplitude long waves in a rotating fluid of finite
depth, Ostrovsky \cite{Ostrovsky} derived the following non-linear
evolution equation
\begin{equation}\label{eq:O}
  \partial_{x}(\partial_{t} u + \partial_{x} f(u) -
  \beta \partial_{xxx}^{3}u) = \gamma u,
\end{equation}
where $\beta$ and $\gamma>0$ are real constants and $f(u) =
\frac{u^{2}}{2}$. Here $u=u(t,x)$ denotes the amplitude of waves,
while $x$ and $t$ are the space and time variables respectively. The
equation can be formally deduced using two asymptotic expansions of
the shallow water equations, once with respect to the rotation
frequency and then with respect to the amplitude of the waves, see
\cite{HunterI}. Later, in a study of long internal waves in a rotating
fluid Hunter \cite{HunterI}, investigated the limit of no
high-frequency dispersion $\beta\to 0$.  This formally reduces
$\eqref{eq:O}$ to the Ostrovsky-Hunter (OH) equation:
\begin{equation}\label{eq:OH}
  \partial_{x}\left(\partial_{t} u + \partial_{x} f(u)\right) = \gamma u
\end{equation}
The OH equation also arises as a model of high frequency waves in a
relaxing medium, see \cite{VakhnenkoI}. In both cases $f(u) =
\frac{u^{2}}{2}$.

Equation \eqref{eq:OH} can also be derived by including the effects of
background rotation in the shallow water equation, and then using
singular perturbation methods, see \cite{RuvoThesis, HunterII}. In
this context, it is worth mentioning that equation \eqref{eq:O}
generalizes the KdV equation, which corresponds to $\gamma = 0$. The
equation \eqref{eq:OH} is also known as the reduced Ostrovsky equation
\cite{Ostrovsky, ParkesI, Stepanyants}, short wave equation
\cite{HunterI}, Ostrovsky-Vakhnenko equation \cite{Boutet, Brunelli},
or Vakhnenko equation \cite{VakhnenkoII, VakhnenkoIII,
  ParkesII}. Also, equation \eqref{eq:O} is used to model ultra short
light pulses in silica optical fibres \cite{AVB, Liu, Schafer,
  GiuseppeIII}, in which case $f(u) = - \frac{1}{6}u^{3}$. In this
case \eqref{eq:O} is sometimes referred to as the
``short-pulse-equation''.

In this context we note that Hunter established the connection between
the KdV equation and short wave equation \eqref{eq:OH}, see
\cite{HunterI}, as the no-rotation and no-long wave dispersion limits
of the same equation. But in the case of oceanic waves near the shore, the
waves usually propagate on a background whose properties vary. In such
a variable medium the linear phase speed of the wave, which is encoded
in the flux term $f(x,u)$ (instead of $f(u)$) has \textit{spatial
  dependency}. To model such scenario the variable coefficient KdV
equation was derived by Johnson \cite{Johnson} for water waves and by
Grimshaw \cite{GrimshawI} for internal waves (see also
\cite{GrimshawII} for a review). Motivated by this, in this paper we
aim to design and analyze a numerical scheme for the OH equation with
spatial dependency in the flux.

As is commonly done, we rewrite the OH equation \eqref{eq:OH} as
the following system
\begin{equation*} 
  u_{t} + f(x,u)_{x} = \gamma P, \quad P_{x} = u.
\end{equation*}
Without loss of generality, we can set $\gamma=1$, and will in the
sequel do so.  Since $P$ is defined as any anti-derivative of $u$, we
need an additional constraint to close the system. This can be done in
different ways, see \cite{GiuseppeI, GiuseppeII, GiuseppeIV}. We shall
adopt the approach in \cite{GiuseppeVI}, where we study the problem
\eqref{eq:OH} in a periodic setting $x\in [0,1]$. In this case it is
natural to redefine the right hand side by subtracting the (constant)
term $\int_0^1 P(t,x)\,dx$. This has the attractive side effect of
conserving $\int_0^1 u\,dx$, i.e.,
\begin{equation*}
  \frac{d}{dt}\int_0^1 u(t,x)\,dx = 0,
\end{equation*}
if $u$ satisfies the \emph{zero mean condition} $\int_0^1 u\,dx=0$.
This condition was also assumed to hold initially in
\cite{Liu,HunterI}.  So the equation we are studying in this paper
reads
\begin{equation}
  \label{eq:OHbrief}
  u_t + f(x,u)_x = P(t,x) - \int_0^1 P(t,y)\,dy,
\end{equation}
where $x\mapsto u(t,x)$ is periodic with period $1$, and
$P(t,x)=\int_0^x u(t,y)\,dy$. Note that since $u$ satisfies the zero
mean condition, $x\mapsto P(t,x)$ is also periodic. This is
essentially an extension of the system studied in \cite{GiuseppeVI} to
the case where $f$ is allowed to depend on the spatial variable
$x$. As in \cite{GiuseppeVI}, discontinuites in $u$ will develop
independently of the smoothness of the initial data, so that
\eqref{eq:OHbrief} must be interpreted in the weak sense. Furthermore,
as with scalar conservation laws, in order to show well posedness, we
shall consider entropy solutions.

Our main results are as follows. Assuming that the mapping $x\mapsto
f(x,v)$ is uniformly Lipschitz continuous locally in $v$, and that the
initial data are of bounded variation and satisfy the zero mean
condition, we have that
\begin{equation*}
  \norm{u(t,\cdot)-v(t,\cdot)}_{L^1((0,1))} 
  \le e^{2t} \norm{u_0-v_0}_{L^1((0,1))},
\end{equation*}
where $u$ and $v$ are entropy solutions with initial data $u_0$ and
$v_0$ respectively. Furthermore, we establish convergence to the unique entropy
solution of approximate solutions generated by an upwind scheme. We
also prove a Kuznetsov-type lemma, see \cite{Kuznetsov}, satisfied by
the entropy solution, and this lemma allows us to conclude that the
approximate solutions converge at the rate $1/2$ in $L^1$.

The rest of this paper is organized as follows. In
Section~\ref{sec:Notations} we detail precise definitions and
assumptions and notation, as well as the definition of the finite
volume scheme. In Section~\ref{sec:Apriori} we prove the necessary
bounds which imply that the approximate solutions form a strongly
compact family in $C([0,T];L^1((0,1)))$. Furthermore, we prove that
the approximate solutions satisfy an entropy inequality, and this is
used to show that any limit of the approximate solutions is an entropy
solution. In Section~\ref{sec:Kuznetsov} we establish a ``Kuznetsov
type lemma'' enabling us to compare exact entropy solutions with
arbitrary functions. Then this comparison
result is used to show that the approximate solutions ``converge at a
rate''. Finally, in Sections~\ref{sec:Numerics} we exhibit some
concrete numerical results.

\section{Preliminaries and notation}\label{sec:Notations}
The problem we study is the following
\begin{equation}\label{eq:OHprecise}
  \begin{gathered}
    u_{t} + f(x,u)_{x} = \int_0^{x} u(t,y)\,dy - \int_0^{1}
    \int_0^{y} u(t,z)\,dzdy, \quad
    \text{for  $x \in (0,1)$,\ $t \in (0,T)$,} \\
    u(0,x) = u_0(x), \quad \text{for $x \in (0,1)$,} \\
    u(t,0) = u(t,1), \quad \text{for $t \in (0,T]$.}
  \end{gathered}
\end{equation}
Regarding the initial data we shall assume that
\begin{equation*}
  u_0 \in BV([0,1]) \quad \text{and} \quad \int_0^{1} u_0(x)\,dx = 0.
\end{equation*}
The flux function $f(x,u)$ is assumed to be in $C^2_{\mathrm{loc}}$,
which in particular implies that it is Lipschitz continuous in $x$ and
locally Lipschitz continuous in $u$. Since solutions of
\eqref{eq:OHprecise} generically develop discontinuities, solutions
must be considered in the weak sense. A function in
$C([0,T];L^1((0,1)))$ is a \emph{weak solution} of
\eqref{eq:OHprecise} if
\begin{equation*} 
  \begin{aligned}
    \int_0^{T} \int_0^{1} u
    \test_{t} & + f(x,u) \test_{x} + P[u] \test \, dx dt \\
    &+ \int_0^{1} u(0,x) \test(x,0) \,dx - \int_0^{1} u(T,x)
    \test(T,x) \, dx = 0,
  \end{aligned}
\end{equation*}
For all test functions $\test=\test(t,x)$ which are $1$-periodic in
$x$.  Here $P[u] := \int_0^{x} u(t,y)\,dy - \int_0^{1}
\int_0^{y} u(t,z)\, dzdy$.  Let $\Pi_T=[0,T]\times[0,1]$, following
\cite{GiuseppeVI} we define entropy solutions as
\begin{definition}[Entropy Solution]\label{def:EntropySoln}
  A function $u \in C([0,T];L^{1}((0,1))) \cap L^{\infty}(\Pi_T)$ is
  called an \emph{entropy solution} of the Ostrovsky-Hunter Equation
  \eqref{eq:OHprecise}, if for all constants $k$ the following
  inequality holds
  \begin{equation}
    \label{eq:EntropySoln}
    \begin{aligned}
      \int_{\Pi_T}\!\!  \eta(u,k) \test_{t} + q(x,u,k) \test_{x} +
      &\sgn{u-k} f_{x}(x,k) \test
      + \sgn{u-k} P[u] \test\, dx dt\\
      & - \int_0^{1} \eta(u(t,x),k) \test(t,x)
      \,dx\Bigm|^{t=T}_{t=0} \ge 0,
    \end{aligned}
  \end{equation}
  for all non-negative test functions $\test$ which are $1$-periodic
  in the $x$ variable. Here $\eta$ and $q$ are defined as the
  Kru\v{z}kov entropy and entropy flux respectively,
  \begin{equation*}
    \eta(u,k)=\abs{u-k},\  \ q(x,u,k)=\sgn{u-k}(f(x,u)-f(x,k)).
  \end{equation*}
\end{definition}
Throughout this paper we employ the following convention, $f_x(x,u)$
and $f_u(x,u)$ denote the partial derivatives of $f$ with respect to
$x$ and $u$ respectively. If $u=u(t,x)$ is differentiable, then we
have
\begin{equation*}
  \frac{\partial}{\partial x} f(x,u(t,x))=f_x(x,u)+f_u(x,u) u_x.
\end{equation*}
Furthermore, we use the convention that $C$ denotes a generic positive
constant, whose actual value may change from one occurrence to the
next.

In order to define the numerical scheme, set
\begin{equation*}
  \Dx = \frac{1}{N}, \ \text{and}  \ \Dt = \frac{T}{M+1},
\end{equation*}
where $N$ and $M$ are positive integers, and $T>0$. We also define
$x_{j+1/2}=j\Dx$ for $j=0,\ldots,N$, $x_j=x_{j+1/2}-\Dx/2$ for
$j=1,\ldots,N$ and $t^n=n\Dt$ for $n=0,1,2,\ldots$. We also define the intervals
$I_j=[x_{j-1/2},x_{j+1/2})$ for $j=1,\ldots,N$ and $I^n =
[t^n,t^{n+1})$. In order to define a piecewise constant
approximations, set $I^n_j = I^n\times I_j$.

Next we define the finite volume approximation. Let $F(x,u,v)$ be a
numerical flux which is monotone and consistent, i.e.,
\begin{align*}
  u &\mapsto F(x,u,v)\ \text{is non-decreasing,}\\
  v &\mapsto F(x,u,v)\ \text{is non-increasing,}\\
  F(x,u,u)&=f(x,u).
\end{align*}
In addition we assume that $F$ is differentiable in $x$ and that both
$F_x$ and $F$ are Lipschitz continuous in $u$ and $v$.

We can now define the finite volume scheme. Set $\lambda=\Dt/\Dx$, and
let $u^{n+1}_j$ be defined by
\begin{equation}\label{eq:TheScheme}
  u^{n+1}_{j} = u^{n}_{j} - \lambda \left(F^{n}_{j+1/2} - F^{n}_{j-1/2}\right) + \Dt  P^{n}_{j},
\end{equation}
for $n\ge 0$ and $j=1,\ldots,N$. Here
\begin{equation*}
  F^n_{j+1/2}=F(x_{j+1/2},u^n_j,u^n_{j+1}),
\end{equation*}
where we use periodic boundary conditions $u^n_{N+1}=u^n_1$ and
$u^n_0=u^n_N$. The term $P^n_j$ is defined as
\begin{equation}\label{eq:discretizedP}
  P^{n}_{j} = \Dx \Bigl(\sum_{i=1}^{j-1} u^{n}_{i} + \frac12
  u^{n}_{j}\Bigr)
  - (\Dx)^2 \sum_{\ell=1}^{N} \Bigl(\sum_{i=1}^{\ell-1} u^{n}_{i} + \frac12 u^{n}_\ell\Bigr).
\end{equation}
Finally we define the initial values $u^0_j$ by
\begin{equation}\label{eq:SamplingInitial}
  u^{0}_{j} = \frac{1}{\Dx} \int_{I_{j}} u_0(x)\, dx, \ \text{for $j=1,\ldots,N$.}
\end{equation}
Observe that since $u_0$ is assumed to be of bounded variation,
\begin{equation*}
  \abs{u^0}_{BV}:=\sum_{j=1}^N \abs{u^0_j-u^0_{j-1}} \le \abs{u_0}_{BV([0,1])}<\infty.
\end{equation*}
The spatial discretization $\Dx$ and the temporal $\Dt$ are related
through a CFL-condition. Consider the map
\begin{equation*}
  \Psi_j(u,v,w)=u-\lambda(F(x_{j+1/2},v,w)-F(x_{j-1/2},u,v)).
\end{equation*}
We choose $\lambda$ so small that for all $j$, $\Psi_j$ is
non-decreasing in all its arguments. For this monotonicity to hold it
is sufficient to choose
\begin{equation}
  \label{eq:CFL}
  \Dt \le c_{f} \Dx,
\end{equation}
where $c_f$ is a constant depending on $f$ (through $F$).

It is also useful to define
\begin{equation*}
  \Dm a_j = \frac{a_j-a_{j-1}}{\Dx}, \ \text{and}\ \Dpt a^n = \frac{a^{n+1}-a^n}{\Dt},
\end{equation*}
where $a_j$ and $a^n$ are any sequences. With this notation the scheme
can be written
\begin{equation*}
  \Dpt u^n_j + \Dm F^n_{j+1/2} = P^n_j.
\end{equation*}
Using \eqref{eq:discretizedP} we see that $\sum_{j=1}^N P^n_j = 0$,
and that it is consistent to define $P^n_0=P^n_N$ and
$P^n_{N+1}=P^n_1$. With this convention we also have that
\begin{equation*}
  \Dm P^n_j = \frac12 \left(u^n_j+u^n_{j-1}\right).
\end{equation*}
We also observe that
\begin{equation*}
  \Dpt \sum_{i=1}^N u^n_j = \sum_{i=1}^N P^n_j = 0.
\end{equation*}
so that if $\int_0^1 u_0\,dx = 0$, then also
\begin{equation*}
  \Dx \sum_{i=1}^N u^n_j = 0\ \text{for $n\ge 0$.}
\end{equation*}
Defining $\norm{u^n}_\infty :=\max_j| u^n_j |$, we can estimate $P^n_j$
as
\begin{equation*} 
  \abs{P^n_j} \le N\Dx \norm{u^n}_\infty + N^2\Dx^2 \norm{u^n}_\infty
  =2\norm{u^n}_\infty.
\end{equation*}
We shall often use the short hand notations
$F_{j+1/2}(u,v)=F(x_{j+1/2},u,v)$, $\dm a_j = a_j - a_{j-1} = \Dx\Dm
a_j$ and $\norm{\cdot}_1 = \norm{\cdot}_{L^1((0,1))}$.

\section{Discrete  estimates and convergence}\label{sec:Apriori}
In this section our aim is to prove the compactness of our scheme
using Kolmogorov's compactness theorem. To employ this theorem we
require a supremum bound, a $BV$ bound and an $L^1$ continuity-in-time
bound on the approximate solutions, all of which are uniform in the
discretization variable $\Dx$.

For simplicity of exposition, we shall show such estimates in the case
where $f_u(x,u)\ge 0$. This means that $F(x,u,v)=\tilde{F}(x,u)$,
which is non-decreasing in $u$. The proof in the general case then
follows \emph{mutatis mutandi}.
\begin{lemma}[$L^{\infty}$-bound]\label{lm:Linftybound}
  The solution $u^{n}_j$ of the scheme \eqref{eq:TheScheme} satisfies
  the following bound
  \begin{equation*} 
    \norm{u^{n}}_{\infty} \le e^{2t^n} \norm{u^{0}}_{\infty}.
  \end{equation*}
\end{lemma}
\begin{proof}
  Using the monotonicity of $\Psi_j$
  \begin{align*}
    u^{n+1}_j &=u^n_j - \lambda \dm F^n_{j+1/2}
    + \Dt \abs{P^n_j}=\Psi_j(u^n_{j-1},u^n_j,u^n_{j+1}) + \Dt P^n_j  \\
    &\le \left(1+2\Dt\right)\norm{u^n}_\infty.
  \end{align*}
  The result follows by an application of Gronwall's inequality.
\end{proof}
Now in the next lemma we are going to obtain a uniform bound on total
variation in space for the numerical solutions.
\begin{lemma}[$BV$ bound]\label{lm:TVbound}
  The solution $u^n$ of the scheme \eqref{eq:TheScheme} satisfies the
  bound
  \begin{equation}\label{eq:TVbound}
    \abs{u^n}_{BV([0,1])} \le e^{C_f t^n}\abs{u^0}_{BV([0,1])} +
    C_f\left(e^{C_ft^n}-1\right),  
  \end{equation}
  where $C_f$ is a positive constant depending on $f$ and its first
  and second derivatives.
\end{lemma}
\begin{proof}
  Assume that $v^n_j$ satisfies
  \begin{align*}
    v^{n+1}_j &= v^n_j - \lambda\dm G^n_{j+1/2}+ \Dt R^n_j \\
    &= v^n_j - \lambda \dm F^n_{j+1/2} + \lambda \dm
    \left((F-G)_{j+1/2}(v^n_j)\right) + \Dt R^n_j,
  \end{align*}
  where $G(x,v)$ is a given function and $\dm = \Dx\Dm$. Using the
  monotonicity of $\Psi_j$ we find that
  \begin{align*}
    \mi{u^n_j}{v^n_j} - \lambda \dm
    F_{j+1/2}&\left(\mi{u^n_j}{v^n_j}\right) + \Dt P^n_j\\
    & \le u^{n+1}_j \le \mx{u^n_j}{v^n_j} - \lambda \dm
    F_{j+1/2}\left(\mx{u^n_j}{v^n_j}\right) + \Dt P^n_j,
  \end{align*}
  and similarly
  \begin{align*}
    &\mi{u^n_j}{v^n_j} - \lambda \dm
    F_{j+1/2}\left(\mi{u^n_j}{v^n_j}\right) +\lambda \dm \left((F-G)_{j+1/2}(v^n_j)\right) + \Dt R^n_j\\
    &\quad \le v^{n+1}_j \le \mx{u^n_j}{v^n_j} - \lambda \dm
    F_{j+1/2}\left(\mx{u^n_j}{v^n_j}\right) + \lambda \dm
    \left((F-G)_{j+1/2}(v^n_j)\right) + \Dt R^n_j.
  \end{align*}
  Subtracting, we find that
  \begin{align}
    \label{eq:neededforentropy}
    u^{n+1}_j - v^{n+1}_j
    &\le \abs{u^n_j-v^n_j} - \lambda \dm Q_{j+1/2}(u^n_j,v^n_j) \\
    &\qquad - \notag \lambda \dm \left((F-G)_{j+1/2}(v^n_j)\right) +
    \Dt\left(P^n_j-R^n_j\right),\\
    \intertext{and} v^{n+1}_j - u^{n+1}_j \label{eq:neededforentropy2}
    &\le \abs{u^n_j-v^n_j} - \lambda \dm Q_{j+1/2}(u^n_j,v^n_j) \\
    &\qquad + \notag \lambda \dm \left((F-G)_{j+1/2}(v^n_j)\right) -
    \Dt\left(P^n_j-R^n_j\right),
  \end{align}
  with
  \begin{equation*} 
    \begin{aligned}
      Q_{j+1/2}(u,k)&=F_{j+1/2}(\mx{u}{k})-F_{j+1/2}(\mi{u}{k}) \\
      &=\sgn{u-k}\left(F_{j+1/2}(u)-F_{j+1/2}(k)\right).
    \end{aligned}
  \end{equation*}
  This implies
  \begin{equation}
    \begin{aligned}
      \abs{u^{n+1}_j - v^{n+1}_j}
      &\le \abs{u^n_j-v^n_j} - \lambda \dm Q_{j+1/2}(u^n_j,v^n_j) \\
      &\qquad - \lambda \abs{\dm \left((F-G)_{j+1/2}(v^n_j)\right)} +
      \Dt\abs{P^n_j-R^n_j}.
    \end{aligned}\label{eq:tvstart}
  \end{equation}
  Regarding the ``$F-G$'' term
  \begin{equation*}
    \abs{\dm \left((F-G)_{j+1/2}(v^n_j)\right)}
    \le \abs{\Delta_-^v(F-G)_{j+1/2}(v^n_j)}
    + \abs{\Delta_-^x(F-G)_{j+1/2}(v^n_{j-1})},
  \end{equation*}
  where
  \begin{align*}
    \Delta_-^v H(x_{j+1/2},w_j)
    &=H(x_{j+1/2},w_j)-H(x_{j+1/2},w_{j-1})\\
    \Delta_-^x H(x_{j+1/2},w_j)&=H(x_{j+1/2},w_j)-H(x_{j-1/2},w_j).
  \end{align*}
  This gives
  \begin{align*}
    \abs{\dm \left((F-G)_{j+1/2}(v^n_j)\right)} &\le \max_j
    \norm{F_{j+1/2}-G_{j+1/2}}_{\mathrm{Lip}} \abs{u^n_j -
      v^n_j}\\
    &\qquad + \abs{\Delta_-^x(F-G)_{j+1/2}(v^n_{j-1})}.
  \end{align*}
  Now we set $v^n_j=u^n_{j-1}$, then $G^n_{j+1/2}=F^n_{j-1/2}$ and
  $R^n_j = P^n_{j-1}$. Furthermore,
  \begin{equation*}
    \abs{\Delta_-^x(F-G)_{j+1/2}(v^n_{j-1})} \le \Dx^2
    \max_{\overset{x\in [0,1]}{\abs{v}\le
        \norm{u^n}_\infty}} \abs{\partial^2_{xx} f(x,v)}.
  \end{equation*}
  Then we get the $BV$ bound
  \begin{align*}
    \sum_{j=1}^N \abs{u^{n+1}_j-u^{n+1}_{j-1}}
    &\le \sum_{j=1}^N \abs{u^n_j - u^n_{j-1}} \\
    &\quad + \Dt \frac{\max_j\lipnorm{F_{j+1/2}-F_{j-1/2}}}{\Dx}
    \sum_{j=1}^N
    \abs{u^n_j - u^n_{j-1}} \\
    &\quad + \Dt \norm{\partial^2_{xx} f}_\infty + \Dt
    \norm{u^n}_\infty
    \\
    &\le \left(1+\Dt\norm{\partial^2_{xu}f}_\infty\right) \sum_{j=1}^N
    \abs{u^n_j - u^n_{j-1}} \\
    &\quad +\Dt \norm{\partial^2_{xx} f}_\infty +\Dt
    \norm{u^n}_\infty.
  \end{align*}
  The estimate \eqref{eq:TVbound} follows after applying
  Lemma~\ref{lm:Linftybound} and then Gronwall's inequality.
\end{proof}
Next, we show a so-called ``discrete entropy inequality''.
\begin{lemma}
  \label{lm:discreteentropyinequality}
  For all $n\ge 0$ and all constants $k$, we have
  \begin{equation}
    \label{eq:discreteKruzkov}
    \begin{aligned}
      \Dpt \eta(u^n_j,k) + \Dm Q_{j+1/2}(u^n_j,k) + \sgn{u^{n+1}_j-k}&
      \Dm F_{j+1/2}(k) \\
      &\le \sgn{u^{n+1}_j-k} P^n_j.
    \end{aligned}
  \end{equation}
\end{lemma}
\begin{proof}
  Choose $R^n_j=0$, $G=0$ and $v^n_j=v^{n+1}_j=k$ in
  \eqref{eq:neededforentropy} and \eqref{eq:neededforentropy2}, the
  results are
  \begin{equation*}
    c \le a + b, \ \text{and} \ -c \le a-b,
  \end{equation*}
  with
  \begin{equation*}
    \begin{gathered}
      d=u^{n+1}_j-k, \ \ a=\abs{u^n_j-k} - \lambda \dm
      Q_{j+1/2}(u^n_j,k) \\ \text{and}\ \ b = -\lambda\dm F_{j+1/2}(k)
      + \Dt P^n_j.
    \end{gathered}
  \end{equation*}
  Note that $a\ge 0$, multiplying the first inequality with $H(c)$ and
  the second with $H(-c)$, where $H$ is the Heaviside function, yields
  \begin{equation*}
    H(c)c\le H(c)a + H(c)b\ \ \text{and}\ \ -H(-c)c\le H(-c)a - H(-c)b.
  \end{equation*}
  Add these and divide by $2$, then rearrange and divide by $\Dt$ to
  get \eqref{eq:discreteKruzkov}.
\end{proof}
\begin{lemma}[Time continuity bound]\label{lm:TimeContinuitybound}
  For all $n\ge 0$ we have that
  \begin{equation}\label{eq:TimeContiBound}
    \Dx \sum_{j=1}^N \abs{\Dpt u^n_j}\le \Dx \sum_{j=1}^N \abs{\Dpt
      u^0_j} + 2\left(e^{2t^n}-1\right)\norm{u^0}_\infty,
  \end{equation}
  where $C_f$ is a constant depending on $f$ and its derivatives.
\end{lemma}
\begin{proof}
  Using $v^n_j=u^{n-1}_j$ in \eqref{eq:tvstart}, we find that $G=F$
  and $R^n_j=P^{n-1}_j$. Thus
  \begin{equation*}
    \sum_{j=1}^N \abs{u^{n+1}_j-u^n_j} \le \sum_{j=1}^N
    \abs{u^n_j-u^{n-1}_j} + \Dt  \abs{P^n_j-P^{n-1}_j}.
  \end{equation*}
  Multiplying with $\lambda$ and using the bound $\abs{P^n_j}\le
  2\norm{u^n}$ and Lemma~\ref{lm:Linftybound}, we get
  \begin{equation*}
    \Dx\sum_{j=1}^N\abs{\Dpt u^n_j} \le \Dx\sum_{j=1}^N\abs{\Dpt
      u^{n-1}_j} + 4\Dt\,e^{2t^n}\norm{u^0}_\infty.
  \end{equation*}
  We use Gronwall's inequality to conclude the proof.
\end{proof}
Note that our assumptions on $f$ and the initial data imply that
\begin{equation}
  \label{eq:timecontinitial}
  \Dx \sum_{j=1}^N \abs{\Dpt u^0_j} \le C_f\left(\abs{u_0}_{BV([0,1])}+1\right)<\infty, 
\end{equation}
for some constant depending on $f$.

Next we define the piecewise constant approximation $u_{\Dx}$ by
\begin{equation}\label{eq:udeltadef}
  u_{\Dx}(t,x) = \sum_{j,n} u^n_j \chi_{I^n_j}(t,x),
\end{equation}
With these three bounds, Lemmas~\ref{lm:Linftybound},
\ref{lm:TVbound}, \ref{lm:TimeContinuitybound}, we can apply Helly's
theorem, \cite[Theorem~A.11]{NilsIII}, to prove that
$\seq{u_{\Dx}}_{\Dx>0}$ is compact.
\begin{lemma}[Compactness lemma]\label{lm:compactnessresult}
  Let $\seq{u_{\Dx}}_{\Dx>0}$ be the family obtained from the scheme
  \eqref{eq:TheScheme} with $\lambda$ chosen such that
  $\Psi_j(u^n_{j-1},u^n_j,u^n_{j+1})$ is monotone for all $j$ and for
  all $t^n<T$. Then the exists a sequence $\seq{\Dx_k}_{k=1}^\infty$
  with $\Dx_k\to 0$ as $k\to \infty$, and a function $u \in
  C([0,T];L^1(0,1))$ such that $u_{\Dx_k} \to u$ in
  $C([0,T];L^{1}(0,1))$.
\end{lemma}
Now we can use the discrete entropy condition
\eqref{eq:discreteKruzkov} to show that any limit $u$ satisfies the
entropy condition \eqref{eq:EntropySoln}.
\begin{theorem}\label{thm:existence}
  Assume that the initial data $u_0 \in BV([0,1])$ satisfies the
  zero mean condition $\int_0^{1} u_0\, dx=0$, and that $\lambda$
  satisfies \eqref{eq:CFL} (so that $\Psi_j$ is monotone). Then
  $u=\lim_{k\to\infty} u_{\Dx_k}$ is an entropy solution according to
  Definition~\ref{def:EntropySoln}.
\end{theorem}
\begin{proof}
  For simplicity we write $\Dx$ for $\Dx_k$.  Choose a non-negative,
  $x$-periodic test function $\test$ and set
  $\test^n_j=\test(t^n,x_j)$. Let $T=t^M$, multiply the discrete
  entropy inequality \eqref{eq:discreteKruzkov} with
  $\Dt\Dx\test^n_j$, and sum by parts in $n$ and $j$ to obtain
  \begin{subequations}
    \label{eq:entr}
    \begin{align}
      \label{eq:entr1}
      \Dt\Dx &\sum_{n=1}^{M-1}\sum_{j=1}^N \eta^n_j \Dmt \test^n_j
      +\Dt\Dx \sum_{n=0}^{M-1} \sum_{j=1}^N Q_{j+1/2}(u^n_j,k) \Dm
      \test^n_j\\
      \label{eq:entr3}
      &\qquad +\Dx\sum_{j=1}^N \eta^M_j \test^{M-1}_j - \eta^0_j
      \test^0_j \\
      \label{eq:entr4}
      &\qquad +\Dt\Dx \sum_{n=0}^{M-1} \sum_{j=1}^N \eta^{\prime,n+1}_j \Dm
      F_{j+1/2}(k) \test^n_j \\
      \label{eq:entr5}
      &\qquad +\Dt\Dx \sum_{n=0}^{M-1} \sum_{j=1}^N
      \eta^{\prime,n+1}_j P^n_j\test^n_j \\
      \notag
      &\ge 0,
    \end{align}
  \end{subequations}
  where
  \begin{equation*}
    \eta^{\prime,n+1}_j = \sgn{u^{n+1}_j-k}, \ \ \text{and}\ \ \Dmt \test^n_j =
    \frac{\test^n_j-\test^{n-1}_j}{\Dt}. 
  \end{equation*}
  Using similar arguments to those that can be found in the proof of
  the analogous result in \cite{GiuseppeVI}, it is straightforward to
  show that we can let $\Dx\downarrow 0$ in \eqref{eq:entr} to
  conclude that $u$ satisfies \eqref{eq:EntropySoln}.  The proof of
  this uses in particular that
  \begin{equation*}
    \begin{gathered}
      \test\in C^2, \ \
      u_{\Dx}(t,\cdot)\in BV, \\
      \text{and}\ \ \norm{u_{\Dx}(t,\cdot)-u_{\Dx}(s,\cdot)}_1\le
      \mathcal{O}(\max\seq{|t-s|,\Dt}),
    \end{gathered}
  \end{equation*}
  and that both $Q$ and $F$ are consistent and Lipschitz continuous in
  both $x$ and $u$.
\end{proof}

\section{A Kuznetsov type lemma, stability and convergence rate}\label{sec:Kuznetsov}
As in \cite{GiuseppeVI} the similarity of the OH equation to a scalar
conservation law allows us to estimate the $L^1$-difference between
the an entropy solution and other functions which are not necessarily
solutions of \eqref{eq:OHprecise}. In this section we establish such a
comparison result and use it to prove that the approximations defined
by the finite volume scheme \eqref{eq:TheScheme} --
\eqref{eq:SamplingInitial} converge to an entropy solution as
$\mathcal{O}(\sqrt{\Dx})$.
\subsection{A comparison result}\label{subsec:Kuznetsovlemma}
For any function $u \in L^\infty([0,T];L^{1}(0,1))$ define
\begin{equation}\label{eq:DoublingFunctionalI}
  \begin{aligned}
    L(u,k,\test) = \int_{\Pi_T} &
    \Bigl(\eta(u,k) \test_{t} + q(x, u, k) \test_{x} \\
    &- \eta'(u,k) f_{x}(x,k) \test
    + \eta'(u,k) P[u] \test \Bigr)\, dxdt \\
    & - \int_0^{1} \eta(u(t,x),k) \test(t,x)\, dx\Bigm|_{t=0}^{t=T},
  \end{aligned}
\end{equation}
where $\eta$ and $q$ are defined in Definition~\ref{def:EntropySoln}
and $\eta'(u,k)=\sgn{u-k}$. Choose the test function
\begin{equation*}
  \test_{\epsilon,\epsilon_0}(t,x,s,y) = \theta_{\epsilon}(x-y)\omega_{\epsilon_0}(t-s),
\end{equation*}
where $\omega_{\epsilon_0}$ and $\theta_{\epsilon}$ are standard
mollifiers, with $\theta_{\epsilon}$ being extended periodically
outside the interval $(-1/2,1/2)$. Let $v \in L^\infty([0,T];L^{1}(0,1))$ and
put $k = v(s,y)$ in $L$ and integrate in $s$ and $y$. This defines the
following functional
\begin{equation*} 
  \Lambda_{\epsilon,\epsilon_0}(u,v) = \int_{\Pi_T}
  L\left(u,v(s,y),\test_{\epsilon,\epsilon_0}(\cdot,\cdot,s,y)\right)
  \,dyds.
\end{equation*}
For any function $w \in C([0,T];L^{1}(0,1))$, we define the moduli of
continuity
\begin{equation*}
  \begin{aligned}
    \mu(w(t,\cdot),\epsilon) &= \sup\limits_{|y| \le \epsilon}
    \norm{w(t,\cdot+y) - w(t,\cdot)}_{1}, \\
    \nu_t(w,\epsilon_0) &= \sup\limits_{|s|\le \epsilon_0}
    \norm{w(t+s,\cdot) - w(t,\cdot)}_{1},  \\
    \nu(w,\epsilon_0) &= \sup\limits_{0 \le t \le T}
    \nu_{t}(w,\epsilon_0).
  \end{aligned}
\end{equation*}
\begin{remark}\label{rem:modulusofcontinuity}
  From the proofs of Lemmas~\ref{lm:Linftybound}, \ref{lm:TVbound},
  and \ref{lm:TimeContinuitybound} if follows that
  \begin{equation*}
    \nu(u_{\Dt,\epsilon_0}) \le (\epsilon_0 + \Dt) C_{T},
  \end{equation*}
  \begin{equation*}
    \mu(u_{\Dt}(t,\cdot), \epsilon) \le \epsilon\abs{u_{\Dt}(t,\cdot)}_{BV([0,1])}.
  \end{equation*}
\end{remark}
In this setting, we have the following version of the Kuznetsov lemma.
\begin{lemma}\label{lm:kuznetsov}
  Let $u$ be an entropy solution of the Ostrovsky-Hunter equation with
  the associated initial data $u_0 \in BV([0,1])$. Then for any $v \in
  L^\infty([0,T];L^{1}(0,1)) \cap L^{\infty}(\Pi_T)$ the
  following estimate holds
  \begin{equation}\label{eq:KuznetsovEstimateFinal}
    \begin{aligned}
      \norm{u(T,\cdot) - v(T,\cdot)}_{1} &\le e^{2T}
      \norm{u_0(\cdot) - v(0,\cdot)}_{1} \\
      &\quad + \left(e^{2T}-1\right)
      \\
      &\quad\times \Bigl[ -\Lambda_{\epsilon,\epsilon_0}(v,u)
      \\
      &\hphantom{\times \Bigl[} \quad + \frac12
      \left(\mu(v(T,\cdot),\epsilon) + \mu(u(T,\cdot),\epsilon) +
        \mu(v(0,\cdot),\epsilon) + \mu(u_0,\epsilon)\right)
      \\
      & \hphantom{\times\Bigl[}\quad + C_{f,T} \bigl(\eps +
      T\sup_{t\in[0,T]}\mu(v(t,\cdot),\eps) + \nu(v,\eps_0)\bigr)
      \Bigr],
    \end{aligned}
  \end{equation}
  where $C_{f,T}$ is a constant depending on $T$, $u_0$, $f$ and its first
  derivatives.
\end{lemma}
\begin{proof}
  For simplicity we write $\test$ for $\test_{\eps,\eps_0}$.  Using
  that $\test_x=-\test_y$, $\test_t=-\test_s$ and that
  $\Lambda_{\eps,\eps_0}(u,v)\ge 0$, we add
  $\Lambda_{\eps,\eps_0}(u,v)$ and $\Lambda_{\eps,\eps_0}(v,u)$ to
  compute
  \begin{align}
    \Lambda_{\epsilon,\epsilon_0}(v,u) &\le
    \int_{\Pi_T}\int_{\Pi_T}\sgn{u-v}\bigl[(f(x,u) - f(x,v)) \test_{x}
    - f_{x}(x,v) \test \notag 
    \\
    &\hphantom{\le \int_{\Pi_T}\int_{\Pi_T}\sgn{u-v}\bigl[} \; -
    (f(y,u) - f(y,v)) \test_{x} + f_{y}(y,u) \test\bigr]\, dxdtdy
    ds \label{eq:K1} 
    \\ 
    & \quad -\int_{\Pi_T}\int_0^{1}
    \test(T,x,s,y) \abs{u(T,x) - v(s,y)}\, dxdsdy \label{eq:K2} 
    \\ 
    & \quad + \label{eq:K3} \int_{\Pi_T} \int_0^{1} \test(0,x,s,y)
    \abs{u_0(x) - v(s,y)} \, dxdsdy 
    \\ 
    & \quad - \label{eq:K4}
    \int_0^{1} \int_{\Pi_T} \test(y,x,T,y) \abs{u(x,t) - v(y,T)} \,
    dydtdx 
    \\ & \quad + \label{eq:K5} \int_0^{1} \int_{\Pi_T}
    \test(t,x,0,y) \abs{u(t,x) - v_0(y)}\, dydtdx 
    \\ 
    & \quad + \int_{\Pi_T}\int_{\Pi_T}\sgn{v-u} (P[v] -
    P[u])\test(t,x,s,y) \,dxdtdyds.  \label{eq:K6}
  \end{align}
  As is standard for scalar conservation laws, see e.g.,
  \cite{NilsIII}, the terms \eqref{eq:K2} -- \eqref{eq:K5} can be
  estimated to yield the terms containing the initial and final data,
  plus the term starting with ``$\frac12 (\cdots$'' in
  \eqref{eq:KuznetsovEstimateFinal}. The term \eqref{eq:K6} can be
  overestimated as in \cite{GiuseppeVI} by
  \begin{equation*}
    2\int_0^T \norm{u(t,\cdot)-v(t,\cdot)}_1\,dt +
    2T\nu(v,\eps_0)+\eps T\norm{v}_\infty.
  \end{equation*}
  The term \eqref{eq:K1} can by estimated as in \cite{NilsII} by
  \begin{equation*}
    C_f \bigl(2\eps + T\sup_{t\in[0,T]}\mu(v(t,\cdot)
    ,\eps) + \nu(v,\eps_0)\bigr).
  \end{equation*}
  Collecting these bounds
  \begin{align*}
    \norm{u(T,\cdot) - v(T,\cdot)}_1
    &\le \norm{u_0-v_0}_1  - \Lambda_{\eps,\eps_0}(v,u) \\
    &\quad + \frac12 \left(\mu(v(T,\cdot),\epsilon) +
      \mu(u(T,\cdot),\epsilon) +
      \mu(v(0,\cdot),\epsilon) + \mu(u_0,\epsilon)\right)\\
    &\quad + C_f \bigl(2\eps +\eps T\norm{v}_\infty  + T(\sup_{t\in[0,T]}\mu(v(t,\cdot,\eps)
    +\nu(v,\eps_0))+
    \nu(v,\eps_0)\bigr)\\
    & \quad + 2\int_0^T \norm{u(t,\cdot)-v(t,\cdot)}_1\,dt.
  \end{align*}
  The proof is concluded by applying Gronwall's inequality.
\end{proof}
If $v$ is another entropy solution with initial data $v_0$ (satisfying
the zero mean condition), then $\Lambda_{\eps,\eps_0}(v,u)\ge 0$, and
we can send $\eps$ and $\eps_0$ to zero in
\eqref{eq:KuznetsovEstimateFinal} to prove the following.
\begin{theorem}\label{thm:uniqueness}
  If $u$ and $v$ are entropy solutions of the Ostrovsky-Hunter
  equation \eqref{eq:OHprecise} with initial data $u_0$ and $v_0$
  respectively, then
  \begin{equation*}
    \norm{u(T,\cdot) - v(T,\cdot)}_1 \le e^{2 T} \norm{u_0 - v_0}_1
  \end{equation*}
\end{theorem}
Observe that since the entropy solution is unique, then the whole
sequence $\seq{u_{\Dx}}$, rather than only a subsequence, converges.

\subsection{Convergence rate}\label{subsec:convrate}
We can use Lemma~\ref{lm:kuznetsov} with $v=u_{\Dx}$ to measure the
$L^1$ error of the finite volume scheme.
\begin{lemma}\label{lm:convergencerate}
  Let $u$ be the entropy solution to \eqref{eq:OHprecise} and let
  $u_{\Dx}$ be the piecewise constant interpolation defined by
  \eqref{eq:udeltadef}, where $u^n_j$ is obtained by the scheme
  \eqref{eq:TheScheme} -- \eqref{eq:SamplingInitial}. Assume that $u_0$
  is in $L^1((0,1))\cap BV([0,1])$, and that $f$ is locally bounded,
  $x$-periodic and twice continuously differentiable. Then there
  exists a constant $C_T$, depending on $f$, $u_0$ and $T$, but not
  on $\eps$, $\eps_0$ or $\Dx$, such that
  \begin{equation}\label{eq:OrderEstimateoriginal}
    -\Lambda_{\epsilon,\epsilon_0}(u_{\Dt},u) \le
    C_{T} \Bigl(\Dt + \Dx + \frac{\Dx}{\epsilon} + \frac{\Dt}{\epsilon_0}\Bigr).
  \end{equation}
\end{lemma}
\begin{proof}
  Again we write $\test=\test_{\eps,\eps_0}$, after summation by parts
  we find that 
  \begin{equation*} 
    \begin{aligned}
      L\left(u_\Dx,k,\test\right)&=-\sum_{n=0}^M \sum_{j=1}^N \Dpt
      \eta^n_j \int_{I^n_j} \test^{n+1}\,dxdt -\sum_{n=0}^M
      \sum_{j=1}^N
      \Dm q^n_{j+1/2}\int_{I^n_j} \test_{j+1/2}\,dxdt \\
      &\qquad + \sum_{n=0}^M \sum_{j=1}^N \int_{I^n_j}
      \left(q(x,u_{\Dx})-q(x_{j+1/2},u_{\Dx})\right)\test_x\, dxdt \\
      &\qquad - \sum_{n=0}^M \sum_{j=1}^N \eta^{\prime,n}_j
      \int_{I^n_j} f_x(x,k)\test \,dxdt \\
      &\qquad + \sum_{n=0}^M \sum_{j=1}^N \eta^{\prime,n}_j \int_{I^n_j}
      P\left[u_{\Dx}\right]\test\, dxdt,
    \end{aligned}
  \end{equation*}
  where $L$ is defined in \eqref{eq:DoublingFunctionalI} and
  \begin{equation*}
    \begin{gathered}
      \test^{n+1}=\test(t^{n+1},x),\ \
      \test_{j+1/2}=\test(t,x_{j+1/2}), \ \ 
      \eta^n_j = \abs{u^n_j-k},\\
      q^n_{j+1/2} =
      \sgn{u^n_j-k}(f(x_{j+1/2},u^n_j)-f(x_{j+1/2},k)), \ \
      \text{and}\ \ \eta^{\prime,n}_j= \sgn{u^n_j-k}. 
    \end{gathered}
  \end{equation*}
  Next, multiply the discrete entropy inequality
  \eqref{eq:discreteKruzkov} by $\int_{I^n_j}\test\,dxdt$ and sum over
  $n$ and $j$ to get
  \begin{align*}
    \ell &:= \sum_{n=0}^M \sum_{j=1}^N \Dpt \eta^n_j
    \int_{I^n_j}\test\,dxdt \\
    &\qquad +\sum_{n=0}^M \sum_{j=1}^N \Dm
    Q^n_{j+1/2}\int_{I^n_j}\test\,dxdt\\
    &\qquad +\sum_{n=0}^M \sum_{j=1}^N \eta^{\prime,n+1}_j\Dm
    F_{j+1/2}(k)\int_{I^n_j}\test\,dxdt\\
    &\qquad - \sum_{n=0}^M \sum_{j=1}^N \eta^{\prime,n+1}_j P^m_j
    \int_{I^n_j}\test\,dxdt\\
    &\le 0.
  \end{align*}
  Write $\mathfrak{L}(k)=L\left(u_\Dx,k,\test\right)$, so that $-\mathfrak{L}\le
  -\mathfrak{L}-\ell$. Then
  \begin{align}
    -\mathfrak{L}(k) &\le
    \sum_{n=0}^M \sum_{j=1}^N \Dpt\eta^n_j \int_{I^n_j}
    \left(\test^{n+1}-\test\right)\,dxdt \label{eq:rate1}
    \\
    &\qquad + \sum_{n=0}^M \sum_{j=1}^N \Dm\left(q^n_j -
      Q^n_{j+1/2}\right) \int_{I^n_j} \int_{I^n_j}\test\,dxdt\label{eq:rate2}\\
    &\qquad + \sum_{n=0}^M \sum_{j=1}^N \Dm q^n_{j+1/2} \int_{I^n_j} 
    \left(\test_{j+1/2}-\test\right)\,dxdt\label{eq:rate3}\\
    &\qquad + \sum_{n=0}^M \sum_{j=1}^N \int_{I^n_j} \left(\eta^{\prime,n}
      f_x(x,k) - \eta^{\prime,n+1}\Dm F_{j+1/2}\right) \test\,dxdt\label{eq:rate4} \\
    &\qquad - \sum_{n=0}^M \sum_{j=1}^N \int_{I^n_j}
    \left(\eta^{\prime,n}P[u^n_j] - \eta^{\prime,n+1}P^n_j\right)
    \test\,dxdt\label{eq:rate5}\\
    &\qquad -\sum_{n=0}^M \sum_{j=1}^N \int_{I^n_j}
      \left(q(x,u_{\Dx})-q(x_{j+1/2},u_{\Dx})\right)\test_x\, dxdt.\label{eq:rate6}
  \end{align}
  We need to bound $-\int_{\Pi_T}\mathfrak{L}(u(s,y))\,dyds$. The integral
  of the first term on the right can be bounded as follows,
  \begin{align*}
    \Bigl| \int_{\Pi_T} \text{\eqref{eq:rate1}}\,dyds\Bigr| &=
    \Bigl|\sum_{n=0}^M \sum_{j=1}^N \int_{\Pi_T} \Dpt \eta^n_j
    \int_{I^n_j}
    \left(\omega_{\eps_0}(t^{n+1}-s)-\omega_{\eps_0}(t-s)\right)
    \theta_\eps(x-y)\,dxdtdyds\Bigr| \\
    &\le \sum_{n=0}^M \sum_{j=1}^N \abs{\Dpt u^n_j}
    \int_{\Pi_T}\int_{I^n_j}\int_t^{t^{n+1}}
    \abs{\omega'_{\eps_0}(\sigma-s)}\,d\sigma\,
    \theta_\eps(x-y)\,dxdtdyds\\
    &\le \sum_{n=0}^M \sum_{j=1}^N \abs{\Dpt u^n_j} \frac{C}{\eps_0}
    \int_{I^n_j} \left(t^{n+1}-t\right)\,dxdt\\
    &\le \frac{C}{\eps_0} \sum_{j=0}^M \Dt^2 C_T\\
    &\le C_T \frac{\Dt}{\eps_0},
  \end{align*}
  where we have used \eqref{eq:TimeContiBound} and
  \eqref{eq:timecontinitial}. The terms \eqref{eq:rate2} and \eqref{eq:rate3} can be bounded
  similarly, 
  \begin{equation*}
    \Bigl|\int_{\Pi_T} \text{\eqref{eq:rate2}} +
    \text{\eqref{eq:rate3}}\,dyds\Bigr|\le C_T \frac{\Dx}{\eps}.
  \end{equation*}
  Next, we consider \eqref{eq:rate4}, we split this into a sum of two terms
  \begin{align*}
    \text{\eqref{eq:rate4}}_a &= \sum_{n,j} \int_{I^n_j}
    \eta^{\prime,n+1}_j
    \left(f_x(x,k)-\Dm F_{j+1/2}(k)\right) \test\,dxdt\\
    \text{\eqref{eq:rate4}}_b &= -\sum_{n,j}
    \left(\eta^{\prime,n+1}_j-\eta^{\prime,n}_j\right)
    f_x(x,k)\,dxdt. 
  \end{align*}
  The second of these can be bounded by summation by parts,
  \begin{align*}
    \text{\eqref{eq:rate4}}_b &= -\sum_{j=1}^N \eta^{\prime,M+1}_j
    \int_{I^M_j} f_x(x,k) \test\,dxdt + \sum_{j=1}^N \eta^{\prime,0}_j
    \int_{I^0_j} f_x(x,k) \test\,dxdt \\
    &\qquad + \sum_{n=1}^M \sum_{j=1}^N \eta^{\prime,n}_j
    \int_{I^n_j} f_x(x,k) (\test(t,x,s,y)-\test(t,x-\Dt,s,y) \,dxdt.
  \end{align*}
  The first integral of the first term of this expression can be
  estimated as
  \begin{align*}
    \Bigl| \int_{\Pi_T} \sum_{j=1}^N &\eta^{\prime,M+1}_j
    \int_{I^M_j} f_x(x,k) \test\,dxdtdyds\Bigr|\\
    &\le \norm{f_x}_\infty \sum_{j=1}^N  \Dx\Dt \\
    &\le C \Dt.
  \end{align*}
  The bound on the second term is identical. To bound
  $\text{\eqref{eq:rate4}}_b$ we must bound the integral of the last term,
  \begin{align*}
    \Bigl|\int_{\Pi_T}\sum_{n=1}^M \sum_{j=1}^N &\eta^{\prime,n}_j
    \int_{I^n_j} f_x(x,k) (\test(t,x,s,y)-\test(t,x-\Dt,s,y)
    \,dxdtdyds \Bigr|\\
    &\le \norm{f_x}_\infty \sum_{n=1}^M\sum_{j=1}^N  
    \int_{\Pi_T}\int_{I^n_j} \int_{t-\Dt}^t
    \abs{\omega_{\eps_0}(\sigma-s)}\,d\sigma dxdtdyds\\
    &\le \norm{f_x}_\infty \frac{C\Dt}{\eps_0}
    \sum_{n=1}^M\sum_{j=1}^N \Dx\Dt\\
    &\le C_T \frac{\Dt}{\eps_0}.
  \end{align*}
  To bound the integral of $\text{\eqref{eq:rate4}}_a$ we use the
  continuity of $f_x$  and the observation that
  \begin{equation*}
    \dm F_{j+1/2}(k)=\dm F(x_{j+1/2},k,k)) = \dm f(x_{j+1/2},k),
  \end{equation*}
  which implies that
  \begin{equation*}
    \Dm F_{j+1/2}(k)=f_x(\xi_j,k)
  \end{equation*}
  for some $\xi_j\in I_j$.
  Therefore
  \begin{align*}
    \Bigl|\int_{\Pi_T} \text{\eqref{eq:rate4}}_a \,dyds\Bigr|
    &=\Bigl| \int_{\Pi_T} \sum_{n,j} \int_{I^n_j}
    \left(f_x(x,k)-f_x(\xi_j)\right)\test\,dxdtdtds\Bigr|  \\
    &\le \Dx\norm{f_{xx}}_\infty \sum_{n,j}\Dx\Dt\\
    &\le C_T \Dx.
  \end{align*}
  Hence 
  \begin{equation*}
    \Bigl|\int_{\Pi_T} \text{\eqref{eq:rate4}}\,dyds \Bigr| \le 
    C_T \Bigl(\Dx + \frac{\Dt}{\eps_0}\Bigr).
  \end{equation*}
  
  The term \eqref{eq:rate5} is bounded in
  \cite[Section~6.2]{GiuseppeVI} as
  \begin{equation*}
    \Bigl|\int_{\Pi_T} \text{\eqref{eq:rate5}}\,dyds \Bigr| \le 
    C_T\Bigl( \Dt+\Dx + \frac{\Dt}{\eps_0}\Bigr).
  \end{equation*}
  The last term, the integral of \eqref{eq:rate6}, can be bounded using
  the Lipschitz continuity of $q$,
  \begin{align*}
    \Bigl|\int_{Pi_T} \text{\eqref{eq:rate6}}\,dyds\Bigr|&\le
    \norm{q_x}_\infty\int_{\Pi_T} \sum_{n,j} \int_{I^n_j} \left(x_{j+1/2}-x\right)
    \abs{\theta'_\eps(x-y)}\omega_{\eps_0}(t-s)\, dxdtdyds \\
    &\le \frac{C}{\eps} \sum_{n,j}\Dx^2\Dt\\
    &\le C_T \frac{\Dx}{\eps}.
  \end{align*}
  The proof is concluded by collecting the bounds on the integrals of all the terms
  \eqref{eq:rate1} -- \eqref{eq:rate6}. 
\end{proof}
From Lemma~\ref{lm:convergencerate} it easily follows that $u_{\Dx}$
converges at a rate $1/2$.
\begin{theorem}\label{thm:halfrate}
  Let $u$ and $u_{\Dx}$ be as in Lemma~\ref{lm:convergencerate}. then
  \begin{equation*} 
    \norm{u(T,\cdot)-u_{\Dx}(T,\cdot)}_1 \le C_T \sqrt{\Dx},
  \end{equation*}
  where $C_{T}$ is a constant independent of $\Dx$.
\end{theorem}
\begin{proof}
  This result follows by setting $\eps=\eps_0=\sqrt{\Dt}=C\sqrt{\Dx}$
  in \eqref{eq:OrderEstimateoriginal}.  We have that $u_{\Dx}(0,x)$  is defined in
  \eqref{eq:SamplingInitial}, therefore
  $\norm{u_{\Dx}(0,\cdot)-u_0}_1\le C\Dx$ since $u_0$ is in $BV$. To
  conclude the proof apply
  Lemma~\ref{lm:kuznetsov},  and recall that for
  $u_{\Dx}$, all moduli of continuity are uniformly linear in the last
  argument.
\end{proof}

\section{Numerical examples}\label{sec:Numerics}
In this section we complement our theoretical results by two numerical
experiments. Both experiments use the flux function 
\begin{equation*}
  f(x,u)=\frac12 u^2 \exp(\sin(2\pi x)).
\end{equation*}
As far as we know, with $f$ given above, there are no solutions to
\eqref{eq:OHprecise} in closed form. When measuring the accuracy of
the approximations, we therefor use an approximation generated by the
finite volume scheme with a small $\Dx$. We used the Engquist-Osher
numerical flux
\begin{equation*}
  F(x,u,v)=\frac12 \exp(\sin(2\pi x)) \left( (\mx{u}{0})^2 + (\mi{v}{0})^2\right).
\end{equation*}

Our first example uses initial data that coincides with those of the
so-called ``corner wave''. This is a closed form  solution of the
OH-equation with $f=u^2/2$, but not so in our case.
This corner wave initial data is given by
\begin{equation}\label{eq:cornerwave}
  u_0(x) =
  \begin{cases}
    \frac{1}{6} \bigl(x - \frac12\bigr)^2 + \frac16\bigl(x -
    \frac12\bigr) + \frac{1}{36},  &\text{for $x \in [0,1/2)$,} 
    \\ 
    \frac{1}{6} \bigl(x - \frac12\bigr)^2 - \frac16\bigl(x - \frac12\bigr) +
    \frac{1}{36},  &\text{for $x \in [1/2,1]$.}
  \end{cases}
\end{equation}
Figure~\ref{fig:1} shows the initial data, as well as the approximate
solutions at $t=36$, with $\Dx=2^{-8}$ and the reference solution
using $\Dx=2^{-13}$ for $t=36$. 
\begin{figure}[h]
  \centering
  \includegraphics[width=0.8\textwidth]{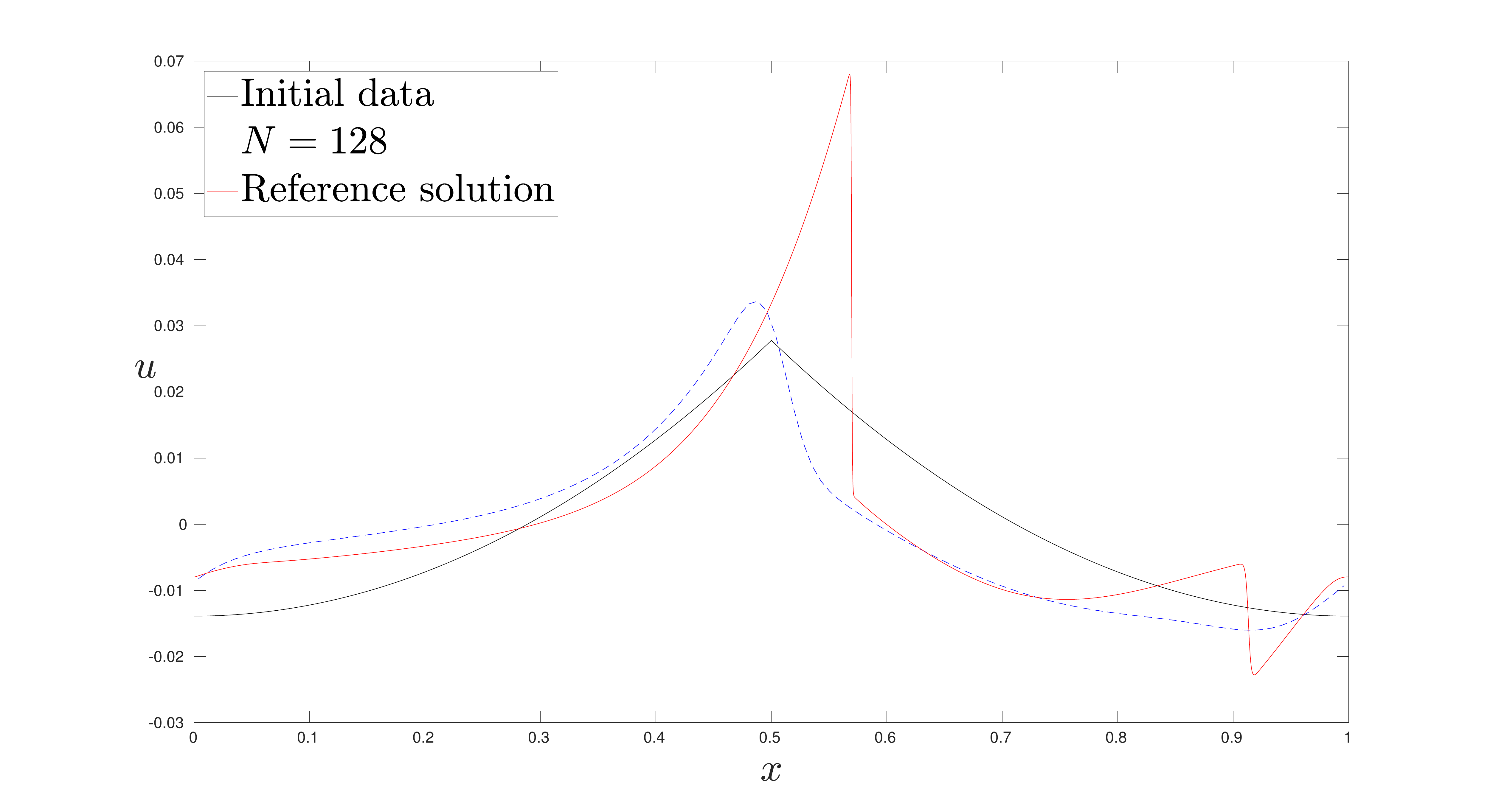}
  \caption{The initial data given by \eqref{eq:cornerwave} and the
      approximate solution with $N=128$ and the reference solution at
      $t=36$.}
  \label{fig:1}
\end{figure}
The second example uses smoother initial data
\begin{equation}
  \label{eq:LiuInitial}
  u_0(x) = -0.05\cos(2\pi x),
\end{equation}
and Figure~\ref{fig:2} shows the approximations for $t=36$, $\Dx=2^{-7}$
and $\Dx=2^{-13}$.
\begin{figure}[h]
  \centering
    \includegraphics[width=0.8\textwidth]{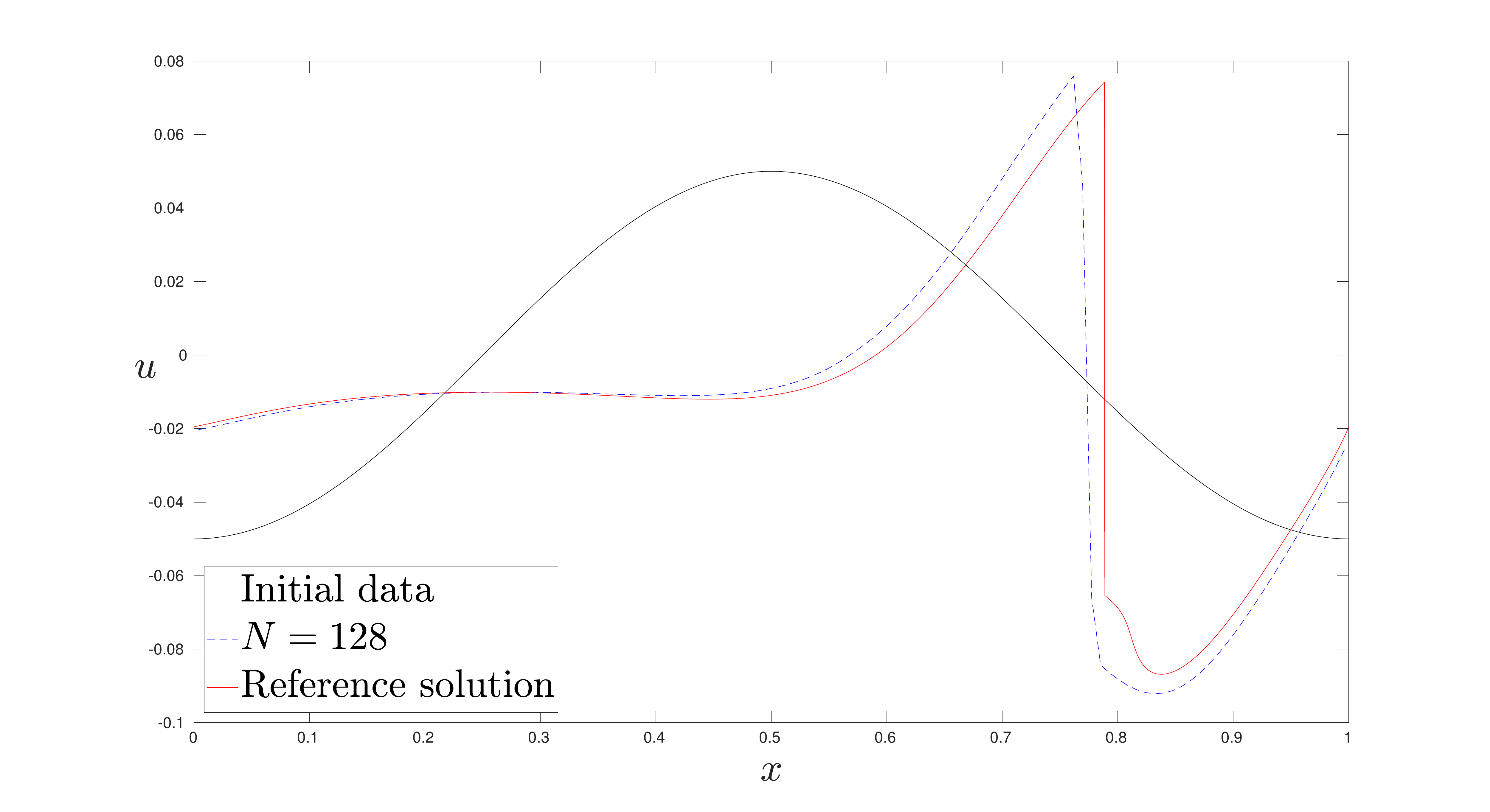}
    \caption{The initial data given by \eqref{eq:LiuInitial} and the
      approximate solution with $N=128$ and the reference solution at
      $t=36$.}
  \label{fig:2}
\end{figure}
We observe that although the data are smooth, the solution seems to
have a discontinuity.

By running the scheme with different $\Dx$, we can try to estimate the
convergence rate numerically. In Table~\ref{tab:1} we show the
relative $L^1$-errors, defined by
\begin{equation*}
  E = 100
  \frac{\norm{u_{\Dt}(t,\cdot)-u_{\mathrm{ref}}(t,\cdot)}_1}
  {\norm{u_{\mathrm{ref}}(t,\cdot)}_1}. 
\end{equation*}
We have done this for both examples, and as
a reference solution, $u_{\mathrm{ref}}$, we used the finite volume
approximation with $\Dx=2^{-13}$.
\begin{table}[h]
  \centering
  \begin{tabular}{r|rr|rr}
    & \multicolumn{2}{c|}{Initial data \eqref{eq:cornerwave}} & \multicolumn{2}{c}{Initial data \eqref{eq:LiuInitial}} \\
    $N$ & \multicolumn{1}{c}{$E$}  &  rate  & \multicolumn{1}{c}{$E$}   & rate \\  \hline
    32  & 56.4 &        & 65.3  &      \\
    64  & 40.9 & 0.5   & 39.5   & 0.7  \\
    128 & 26.4 & 0.6   & 21.8   & 0.9  \\
    256 & 15.5 & 0.8   & 11.4   & 0.9  \\
    512 &  7.8 & 1.0   &  5.9   & 1.0  \\
    1024&  3.7 & 1.1   &  2.7   & 1.1  \\
    2048&  1.6 & 1.2   &  1.2   & 1.2  \\
  \end{tabular}
  \caption{Errors and numerical convergence rate in $L^1$  at $t=36$.}
  \label{tab:1}
\end{table}
We observe that the convergence rates are higher that the
theoretically proven rate. Although we are measuring
``self-convergence'', it may well be the case that when the solution
is a smooth as our examples seem to show (continuously differentiable
except for a single discontinuity), the actual convergence rate is
higher than $1/2$.

\end{document}